\documentclass[11pt]{amsart}
\usepackage[margin=1.15in]{geometry}
\usepackage{amsmath,amssymb,amsthm}
\usepackage{graphicx}
\usepackage{float}
\newtheorem{theorem}{Theorem}[section]
\newtheorem{proposition}[theorem]{Proposition}
\newtheorem{lemma}[theorem]{Lemma}

\theoremstyle{remark}

\theoremstyle{definition}
\newtheorem{example}[theorem]{Example}

\newcommand{\R}{\mathbb R}
\newcommand{\Z}{\mathbb Z}

\newcommand{\one}{\mathbf 1}
\newcommand{\supp}{\operatorname{supp}}

\title{On the Two-Weight Problem for One-Sided Maximal Operators in Higher Dimensions}

\author{Dinghuai Wang}
\address{School of Mathematics and Statistics, Anhui Normal University, Wuhu 241002, China.}
\email{Wangdh1990@126.com}

\vspace{0.5cm}

\begin{document}
\begin{abstract}
For the one-sided Hardy--Littlewood maximal operator $M_d^+$ on
$\mathbb R^d$, the natural two-weight Muckenhoupt condition was shown
by Sawyer in 1986 to characterize the weak $(p,p)$ inequality in
dimension one, and by Forzani, Mart\'in-Reyes and Ombrosi in 2011 in
dimension two. At the endpoint $p=1$ in dimension three, Ombrosi and Nazarov
have recently given negative answers to both the
Fefferman--Stein-type question and the related weighted
weak-type $(1,1)$ question
\cite{OmbrosiNazarov} (personal communication). In this
paper, we prove that the two-weight characterization fails for every
$d\geq 3$ and $p>1$. More precisely, for every
$1<p<\infty$ and every $d\geq 3$, there exist weights $w$ and $v$ such
that
$$
A_{p,d}^+(w,v)<\infty,
\qquad
\|M_d^+\|_{L^p(v)\to L^{p,\infty}(w)}=\infty.
$$
The proof uses a finite two-dimensional Hardy operator whose weak operator norm is bounded below by $c_p(\log N)^{1/p'}$. This operator is embedded
into the three-dimensional lattice maximal operator and then
transferred to the continuous setting. A tensor extension yields the
same failure in every dimension $d>3$.
\vskip 0.2 true cm


\vskip 0.2 true cm
\noindent
\end{abstract}

\maketitle

\section{Introduction}

The Hardy--Littlewood maximal operator is one of the basic objects of harmonic analysis. Notice that the usual Hardy-Littlewood maximal is the two-sided operator
$$
Mf(x)
=
\sup_{h>0}\frac1{2h}\int_{x-h}^{x+h}|f(t)|\,dt.
$$
Its one-sided analogue on the real
line is
$$
M_1^+f(x)
=
\sup_{h>0}\frac1h\int_x^{x+h}|f(t)|\,dt.
$$
Although $M_1^+$ differs from the classical maximal operator only in
the geometry of the averaging intervals, this asymmetry produces a
distinct weighted theory. The one-sided weight classes and the
corresponding weighted inequalities were developed in
\cite{MRNew,MROT,SawyerOneSided}. In particular,
Sawyer~\cite{SawyerOneSided} characterized the two-weight weak-type inequalities. Subsequent work includes Lorentz, Orlicz,
dyadic, and quantitative estimates; see
\cite{LMR,LMRR,MRSharp,Ombrosi, OrtegaLorentz,OrtegaOrlicz}.

The same geometry appears in ergodic theory. For a measure-preserving
flow $\{\tau_t\}_{t\in\mathbb R}$, the associated one-sided ergodic
maximal operator is
$$
M_\tau f(x)
=
\sup_{h>0}\frac1h\int_0^h|f(\tau_t x)|\,dt.
$$
Translation on $\mathbb R$ recovers $M_1^+$, while transference
connects one-sided maximal estimates with maximal inequalities for
ergodic averages. This connection persists in several parameters; in
particular, the two-dimensional theorem of Forzani,
Mart\'in-Reyes and Ombrosi~\cite{FMO} has an ergodic counterpart for
bi-parameter flows.

For $d\ge1$, define
$$
M_d^+f(x)
=
\sup_{h>0}\frac1{h^d}
\int_{x+[0,h]^d}|f(y)|\,dy,
\qquad x\in\mathbb R^d.
$$
Given weights $w,v$ and $1<p<\infty$, the natural problem is to
characterize
$$
M_d^+:L^p(v)\longrightarrow L^{p,\infty}(w).
$$
Write
$$
\sigma=v^{-1/(p-1)}.
$$
For
$$
Q^-=\prod_{j=1}^d[a_j-h,a_j],
\qquad
Q^+=\prod_{j=1}^d[a_j,a_j+h],
$$
set
$$
A_{p,d}^+(w,v)
=
\sup_{a\in\mathbb R^d,\ h>0}
\frac{w(Q^-)\sigma(Q^+)^{p-1}}{h^{dp}}.
$$
It is said that $(w, v)$ satisfies $A_{1,d}^+(w,v)$ if there exists a positive constant $C$ such that
for all $h > 0$, 
$$\frac{1}{h^d}\int_{x+[-h,0]^{d}}w(y)dy\leq Cv(x)$$
for almost every $x$. The weak-type inequality always shows $A_{p,d}^+(w,v)<\infty$.
Thus the issue is sufficiency. The condition $A_{p,d}^+(w,v)$ was shown
by Sawyer in \cite{SawyerOneSided} to characterize the weak $(p,p)$ inequality in
dimension one. In dimension two, Forzani, Mart\'in-Reyes and Ombrosi~\cite{FMO}
proved that for $1\leq p<\infty$,
$$
M_2^+:L^p(v)\longrightarrow L^{p,\infty}(w)
\quad\Longleftrightarrow\quad
A_{p,2}^+(w,v)<\infty.
$$
A central ingredient of their argument is a two-dimensional geometric
covering lemma. They pointed out the difficulty of extending
that geometry to higher dimensions. Related higher-dimensional results
under restricted weak-type or Orlicz hypotheses were obtained in
\cite{Berra,GM}, but they do not resolve the full two-weight
characterization above. Recently, an endpoint obstruction in dimension three has been obtained in \cite{OmbrosiNazarov}. There is also a closely related obstruction at the endpoint $p=1$. Define
$$
M_3^-w(x)
=
\sup_{h>0}\frac1{h^3}
\int_{x+[-h,0]^3}w(y)\,dy.
$$
Two natural questions arise: whether, for every weight $w$,
\begin{equation}\label{eq:ON-endpoint-first}
w\bigl(\{x\in\R^3:M_3^+f(x)>1\}\bigr)
\lesssim
\int_{\R^3}|f|\,M_3^-w,
\end{equation}
and whether, under the additional assumption $M_3^-w\lesssim w$ almost everywhere, one has
\begin{equation}\label{eq:ON-endpoint-second}
w\bigl(\{x\in\R^3:M_3^+f(x)>1\}\bigr)
\lesssim
\int_{\R^3}|f|\,w.
\end{equation}
Ombrosi and Nazarov have recently shown that \emph{both questions have negative answers} \cite{OmbrosiNazarov} (personal communication).

\medskip

In this paper, we prove that the two-weight characterization fails for every
$d\geq 3$ and $p>1$. Our main result can be stated as follows.

\begin{theorem}\label{thm:main}
Let $1<p<\infty$ and $d\ge3$. There exist a nonnegative locally
integrable weight $w$ and a positive locally integrable weight $v$ on
$\mathbb R^d$ such that
$$
A_{p,d}^+(w,v)<\infty
$$
and
$$
\|M_{d}^+\|_{L^p(v)\to L^{p,\infty}(w)}
=
\infty.
$$
\end{theorem}

The proof is carried out first in dimension three. Once the
three-dimensional counterexample is obtained, a tensor extension gives
the result for every $d>3$. Thus the main point is to construct weights
on $\mathbb R^3$ satisfying
$$
A_{p,3}^+(w,v)<\infty,
\qquad
\|M_3^+\|_{L^p(v)\to L^{p,\infty}(w)}=\infty.
$$
Then, we begin with the corresponding discrete problem.

The proof is organized as follows. Section~2 develops the lattice problem and the finite Hardy operator. Section~3 gives the continuous
finite construction. Section~4 shows the global counterexample in
dimension three and proves the extension to all $d>3$.

\section{The lattice problem and the finite Hardy operator}
\subsection{Lattice problem and necessity}
For $n\in\mathbb Z^3$, define the lattice one-sided maximal operator by
\begin{equation*}
M_{\mathbb Z^3}^{+}f(n)
=
\sup_{r\geq1}
\frac{1}{r^3}
\sum_{m\in\{0,\ldots,r-1\}^3}
|f(n+m)|.
\end{equation*}
Discrete Hardy--Littlewood operators on $\mathbb Z^d$ have been studied
extensively, in particular in connection with dimension-free estimates.
Bourgain, Mirek, Stein and Wr\'obel~\cite{BMMSW} considered discrete operators over cubes in $\mathbb Z^d$. They proved
dimension-free maximal estimates for the full family of cube averages
in the range $3/2<p\leq\infty$, and for dyadic scales in the full range
$1<p\leq\infty$. They also obtained variational estimates and related
ergodic consequences.

For $a=(a_1,a_2,a_3)\in\mathbb Z^3$ and $r\geq1$, define
\begin{align*}
Q_{\mathbb Z}^{-}(a,r)
&=
\prod_{j=1}^3
\{a_j-r+1,\ldots,a_j\},
\\
Q_{\mathbb Z}^{+}(a,r)
&=
\prod_{j=1}^3
\{a_j,\ldots,a_j+r-1\}.
\end{align*}
Both cubes contain exactly $r^3$ lattice points and have the common
vertex $a$.

To distinguish discrete weighted sums from their continuous integral counterparts, we define
\begin{equation*}
w[E] := \sum_{n\in E} w(n),
\end{equation*}
while the continuous analogue is denoted by
$$
w(E)=\int_E w(x)\,dx.
$$

Let $w$ and $v$ be weights on $\mathbb Z^3$, with $v>0$, and
$
\sigma=v^{-1/(p-1)}.
$
The associated discrete one-sided Muckenhoupt constant is
\begin{equation*}
A_{p,\mathbb Z^3}^{+}(w,v)
=
\sup_{\substack{a\in\mathbb Z^3\\ r\geq1}}
\frac{
w[Q_{\mathbb Z}^{-}(a,r)]
\,
\sigma[Q_{\mathbb Z}^{+}(a,r)]^{p-1}
}{
r^{3p}
}.
\end{equation*}
We use the weighted weak norm
$$
\|g\|_{\ell^{p,\infty}(w)}
=
\sup_{\lambda>0}
\lambda\,
w\bigl[
\{n\in\mathbb Z^3:|g(n)|>\lambda\}
\bigr]^{1/p}.
$$

We write
\begin{equation*}
x\preceq y
\quad\Longleftrightarrow\quad
x_j\leq y_j,
\qquad j=1,2,3.
\end{equation*}
For every pair $Q^{-}_{\mathbb Z}:=Q_{\mathbb Z}^{-}(a,r)$ and
$Q^{+}_{\mathbb Z}:=Q_{\mathbb Z}^{+}(a,r)$,
\begin{equation*}
x\in Q^{-}_{\mathbb Z},
\qquad
y\in Q^{+}_{\mathbb Z}
\quad\Longrightarrow\quad
x\preceq y. \qquad (\text{see~Figure}~\ref{fig-1})
\end{equation*}

\begin{figure}[H]
\centering
\includegraphics[width=.48\textwidth]{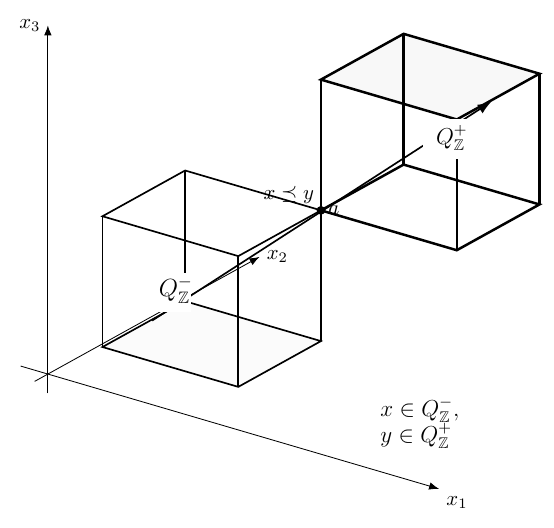}
\caption{$x\in Q^-_{\Z}$ and $y\in Q^+_{\Z}$.}
\label{fig-1}
\end{figure}

The discrete analogue of the necessity statement is as follows.

\begin{theorem}\label{thm:lattice}
Let $1<p<\infty$. If
\begin{equation}\label{eq:lattice-weak}
\|M^+_{\Z^3}f\|_{\ell^{p,\infty}(w)}
\le C\|f\|_{\ell^p(v)},
\end{equation}
then $A_{p,\Z^3}^+(w,v)<\infty$. However, there exist $w\ge0$ and $v>0$ on $\Z^3$ such that
$$
A_{p,\Z^3}^+(w,v)<\infty,
\qquad
\|M^+_{\Z^3}\|_{\ell^p(v)\to\ell^{p,\infty}(w)}=\infty.
$$
\end{theorem}

\begin{proof}[Proof of necessity]
Fix $a\in\Z^3$ and $r\ge1$, assume first that $\sigma[Q^+_{\Z}]<\infty$ and take $f=\sigma\one_{Q^+}$. If $n\in Q^-_{\Z}$ and $m\in Q^+_{\Z}$, then
$$
0\le m_j-n_j\le 2r-2.
$$
Hence
$$
Q^+_{\Z}\subset n+\{0,\ldots,2r-2\}^3,
$$
and
$$
M^+_{\Z^3}f(n)
\ge
\frac{\sigma[Q^+_{\Z}]}{(2r-1)^3},
\qquad n\in Q^-_{\Z}.
$$
The inequality \eqref{eq:lattice-weak} gives
$$
\frac{\sigma[Q^+_{\Z}]}{(2r-1)^3}w[Q^-_{\Z}]^{1/p}
\le
C\left(\sum_{Q^+_{\Z}}\sigma^pv\right)^{1/p}.
$$
Since $\sigma^pv=\sigma$,
$$
w[Q^-_{\Z}]\frac{\sigma[Q^+_{\Z}]^p}{(2r-1)^{3p}}
\le C^p\sigma[Q^+_{\Z}].
$$
Therefore,
$$
w[Q^-_{\Z}]\sigma[Q^+_{\Z}]^{p-1}
\le C^p(2r-1)^{3p}
\le 2^{3p}C^pr^{3p}.
$$
If $\sigma[Q^+_{\Z}]=\infty$, replace $\sigma$ by $\min\{\sigma,k\}$ and let $k\to\infty$. Thus $A_{p,\Z^3}^+(w,v)\le2^{3p}C^p$.
\end{proof}

To disprove the converse, we first construct a finite two-parameter Hardy operator.

\subsection{Finite two-parameter Hardy operator}

Two-dimensional Hardy inequalities already exhibit genuinely
multiparameter two-weight phenomena. Sawyer~\cite{Sawyer} obtained
a characterization for the continuous two-dimensional Hardy operator,
while Rakotondratsimba~\cite{Rakotondratsimba} treated the discrete
two-dimensional operator. More recent work has considered
multi-dimensional discrete Hardy inequalities; see
\cite{StepanovUshakova25}. The finite two-parameter Hardy operator below is not an application
of those characterizations. It is needed for the maximal-function construction.

Fix $N\ge8$ and let
$$
\Lambda_N=\{1,\ldots,N\}^2.
$$
Define $\mu$ on $\Lambda_N$ by
\begin{equation*}
\mu_{ij}=\one_{\{i+j=N+1\}}.
\end{equation*}
Thus $\supp\mu\subset\{(i,j):i+j=N+1\}$. For $(i,j)\in\Lambda_N$, set
\begin{equation*}
A_{ij}
=
\sum_{r\le i}\sum_{s\le j}\mu_{rs}.
\end{equation*}
The intersection of $\{(r,s):r\le i,\ s\le j\}$ with $\supp\mu$ has cardinality $(i+j-N)_+$, see Figure \ref{fig-2}.
\begin{figure}[H]
\centering
\includegraphics[width=.42\textwidth]{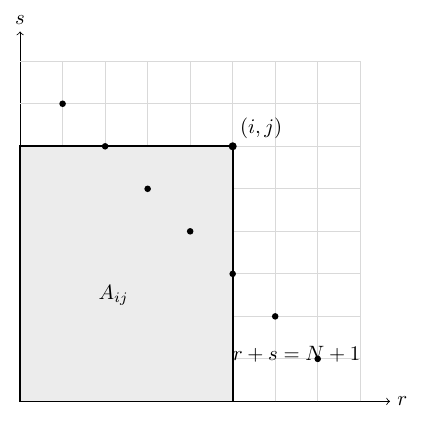}
\caption{The index set $\{(r,s):r\le i,\ s\le j\}$.}
\label{fig-2}
\end{figure}
A direct computation gives
\begin{equation*}
A_{ij}=(i+j-N)_+.
\end{equation*}
Set $\nu_{ij}=0$ when $i+j\le N$. If $m=i+j-N\ge1$, define $b_m=m^{1-p'}, p'=\frac{p}{p-1}$ and
\begin{align*}
\nu_{ij}&=b_m-2b_{m+1}+b_{m+2}, && i<N,\ j<N,\\
\nu_{Nj}&=b_j-b_{j+1}, && j<N,\\
\nu_{iN}&=b_i-b_{i+1}, && i<N,\\
\nu_{NN}&=b_N.
\end{align*}
Since $t\mapsto t^{1-p'}$ is decreasing and convex on $(0,\infty)$, all the numbers are nonnegative.
For $i+j\ge N+1$, let
\begin{equation*}
S_{ij}=\sum_{k=i}^N\sum_{\ell=j}^N\nu_{k\ell}.
\end{equation*}
The two index sets defining $A_{ij}$ and $S_{ij}$ are shown in Figure~\ref{fig-3}.

\begin{figure}[H]
\centering
\includegraphics[width=.42\textwidth]{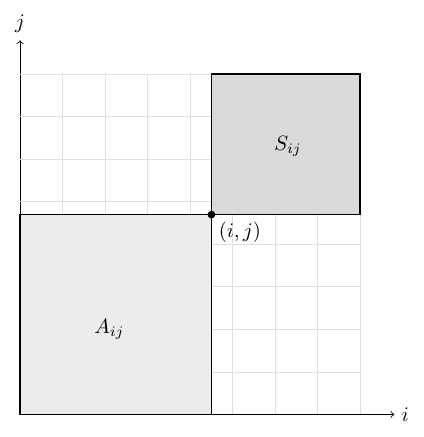}
\caption{The index sets $\{r\le i,\ s\le j\}$ and $\{k\ge i,\ \ell\ge j\}$ defining $A_{ij}$ and $S_{ij}$.}
\label{fig-3}

\end{figure}

\begin{example}\label{ex:N6}
Let $N=6$ and $(i,j)=(4,4)$. Then
$$
m=i+j-N=4+4-6=2,
$$
and $A_{44}=m=2$. On the other hand, the relevant $3\times 3$ block of $\nu_{k\ell}$ for $4\le k,\ell\le 6$ is
$$
\begin{array}{c|ccc}
 & \ell=4 & \ell=5 & \ell=6\\ \hline
k=4 & b_2-2b_3+b_4 & b_3-2b_4+b_5 & b_4-b_5\\
k=5 & b_3-2b_4+b_5 & b_4-2b_5+b_6 & b_5-b_6\\
k=6 & b_4-b_5 & b_5-b_6 & b_6
\end{array}
$$
Summing all terms gives
$$
\begin{aligned}
S_{44}
&=
(b_2-2b_3+b_4)
+(b_3-2b_4+b_5)
+(b_4-b_5)\\
&\quad
+(b_3-2b_4+b_5)
+(b_4-2b_5+b_6)
+(b_5-b_6)\\
&\quad
+(b_4-b_5)
+(b_5-b_6)
+b_6.
\end{aligned}
$$
It follows that
$$
S_{44}=b_2=2^{1-p'}=m^{1-p'} \quad \text{and} \quad A_{44}S_{44}^{p-1}=1.
$$
\end{example}

The same cancellation holds for every $N$ and every index with $i+j\ge N+1$.

\begin{proposition}\label{prop:finite}
Let $1<p<\infty$. For $i+j\ge N+1$,
\begin{equation*}
S_{ij}=(i+j-N)^{1-p'} \quad \text{and} \quad
A_{ij}S_{ij}^{p-1}=1.
\end{equation*}
Moreover,
\begin{equation}\label{eq:NlogN}
c_pN\log N
\le
\sum_{i=1}^N\sum_{j=1}^N A_{ij}^{p'}\nu_{ij}
\le
C_pN\log N.
\end{equation}
\end{proposition}

\begin{proof}
Let
$$
U_N=\{(i,j)\in\Lambda_N:i+j\ge N+1\}
$$
and set
$$
B_{ij}=b_{i+j-N},\qquad (i,j)\in U_N.
$$
and $B_{N+1,j}=B_{i,N+1}=0$. Then the definition of $\nu$ can be written as
\begin{equation*}
\nu_{ij}=B_{ij}-B_{i+1,j}-B_{i,j+1}+B_{i+1,j+1},
\qquad (i,j)\in U_N.
\end{equation*}
We first sum in the second coordinate:
\begin{align*}
\sum_{\ell=j}^N\nu_{k\ell}
&=
\sum_{\ell=j}^N
\bigl(B_{k\ell}-B_{k+1,\ell}-B_{k,\ell+1}+B_{k+1,\ell+1}\bigr)\\
&=B_{kj}-B_{k+1,j}.
\end{align*}
Summing now in $k$, we obtain
$$
S_{ij}
=
\sum_{k=i}^N(B_{kj}-B_{k+1,j})
=B_{ij}
=b_{i+j-N}=(i+j-N)^{1-p'}.
$$
For $m=i+j-N\ge1$, we have $A_{ij}=m$ and $S_{ij}=m^{1-p'}$. Since
$$
(1-p')(p-1)=-1,
$$
we obtain
$$
A_{ij}S_{ij}^{p-1}
=m\,m^{-1}=1.
$$

For $m\geq1$, the mean value theorem gives a point
$\xi_m\in(m,m+1)$ such that
\begin{equation*}
b_m-b_{m+1}
=
(p'-1)\xi_m^{-p'}
\simeq_p
m^{-p'}.
\end{equation*}
Similarly,
\begin{equation*}
b_m-2b_{m+1}+b_{m+2}
\simeq_p
m^{-p'-1}.
\end{equation*}
On the diagonal $i+j=N+m$ we have $A_{ij}=m$. For $1\le m\le N-2$ there are $N-m-1$ interior points, and there are two boundary points.
At the point $(N,N)$, $A_{NN}=N$. Therefore
\begin{equation*}
\sum_{i,j}A_{ij}^{p'}\nu_{ij} = \sum_{m=1}^{N-2}(N-m-1)m^{p'}(b_m-2b_{m+1}+b_{m+2}) +2\sum_{m=1}^{N-1}m^{p'}(b_m-b_{m+1}) +N^{p'}b_N.
\end{equation*}
By the two estimates above,
$$
m^{p'}
(b_m-b_{m+1})
\lesssim_p1,
$$
and
$$
m^{p'}
(b_m-2b_{m+1}+b_{m+2})
\lesssim_p
\frac1m.
$$
Therefore,
\begin{align*}
\sum_{i,j}A_{ij}^{p'}\nu_{ij}
&\lesssim_p
\sum_{m=1}^{N-2}
(N-m-1)\frac1m
+N+N\\
&\lesssim_p
N\sum_{m=1}^{N}\frac1m \lesssim_p
N\log N .
\end{align*}

For the reverse inequality, it is enough to consider
$$
4\le m\le \frac N2 .
$$
Since $N\ge8$,
$
N-m-1\ge \frac N4 .
$
Using the second-difference estimate,
$$
m^{p'}
(b_m-2b_{m+1}+b_{m+2})
\ge
\frac{c_p}{m}.
$$
Consequently,
\begin{align*}
\sum_{i,j}A_{ij}^{p'}\nu_{ij}
&\ge
c_pN
\sum_{m=4}^{\lfloor N/2\rfloor}
\frac1m \ge
c_pN\log N .
\end{align*}
The upper and lower bounds together prove \eqref{eq:NlogN}.
\end{proof}

Define the finite two-dimensional Hardy operator
\begin{equation*}
H_Ng(i,j)=\sum_{k=i}^N\sum_{\ell=j}^N g_{k\ell}\nu_{k\ell}
\end{equation*}
and let
$$
C_N=\|H_N\|_{\ell^p(\nu)\to\ell^{p,\infty}(\mu)}.
$$

\begin{lemma}\label{lem:weak-integration}
If $u\ge0$ on $\Lambda_N$ and $\mu[\Lambda_N]=N$, then
\begin{equation*}
\sum_{\Lambda_N}u\,d\mu
\le
p'\|u\|_{\ell^{p,\infty}(\mu)}N^{1/p'}.
\end{equation*}
\end{lemma}

\begin{proof}
By the distribution formula,
$$
\sum_{\Lambda_N}u\,d\mu
=
\int_0^\infty\mu[\{u>t\}]\,dt.
$$
Then
$$
\mu[\{u>t\}]\le \min\{N,\|u\|_{\ell^{p,\infty}(\mu)}^pt^{-p}\}.
$$
Choose $t_0=\|u\|_{\ell^{p,\infty}(\mu)}N^{-1/p}$. Splitting at $t_0$ yields
\begin{align*}
\int_0^\infty\mu[\{u>t\}]\,dt
&\le Nt_0+\|u\|_{\ell^{p,\infty}(\mu)}^p\int_{t_0}^\infty t^{-p}\,dt\\
&=\|u\|_{\ell^{p,\infty}(\mu)}N^{1/p'}+\frac1{p-1}\|u\|_{\ell^{p,\infty}(\mu)}N^{1/p'}\\
&=p'\|u\|_{\ell^{p,\infty}(\mu)}N^{1/p'}.
\end{align*}
\end{proof}

By Fubini, we have
\begin{align*}
\sum_{i,j}H_Ng(i,j)\mu_{ij}
&=
\sum_{i,j}\sum_{k\ge i,\ell\ge j}g_{k\ell}\nu_{k\ell}\mu_{ij}\\
&=
\sum_{k,\ell}g_{k\ell}\nu_{k\ell}
\sum_{i\le k,j\le\ell}\mu_{ij} =
\sum_{k,\ell}g_{k\ell}A_{k\ell}\nu_{k\ell}.
\end{align*}
Combining Lemma~\ref{lem:weak-integration} with the definition of $C_N$ gives
$$
\sum_{k,\ell}g_{k\ell}A_{k\ell}\nu_{k\ell}
\le
p'C_NN^{1/p'}\|g\|_{\ell^p(\nu)},
$$
since $\mu(\Lambda_N)=\sum_{i=1}^{N}\sum_{j=1}^{N}\mu_{ij}=A_{NN}=N$. Taking the dual supremum over $g$ yields
$$
\left(\sum_{k,\ell}A_{k\ell}^{p'}\nu_{k\ell}\right)^{1/p'}
\le
p'C_NN^{1/p'}.
$$
Proposition~\ref{prop:finite} implies
\begin{equation}\label{eq:CN-lower}
C_N\ge c_p(\log N)^{1/p'}.
\end{equation}
Choose $g_N\ge0$ so that
\begin{equation}\label{eq:gN}
\|H_Ng_N\|_{\ell^{p,\infty}(\mu)}
\ge
\frac12C_N\|g_N\|_{\ell^p(\nu)}.
\end{equation}

\subsection{The three-dimensional lattice counterexample}
The finite two-dimensional Hardy operator contains the obstruction but not yet the
one-sided maximal geometry. We place the two finite measures on parallel
horizontal planes in $\Z^3$ and use one-sided cubes to connect them.
Define
\begin{equation*}
\omega_N(i,j,k)=\mu_{ij}\one_{\{k=0\}},
\qquad
\tau_N(i,j,k)=(4N)^{3p'}\nu_{ij}\one_{\{k=4N\}}.
\end{equation*}
Hence $\supp\omega_N\subset\{k=0\}$ and $\supp\tau_N\subset\{k=4N\}$; see Figure~\ref{fig-4}.

\begin{figure}[H]
\centering
\includegraphics[width=.50\textwidth]{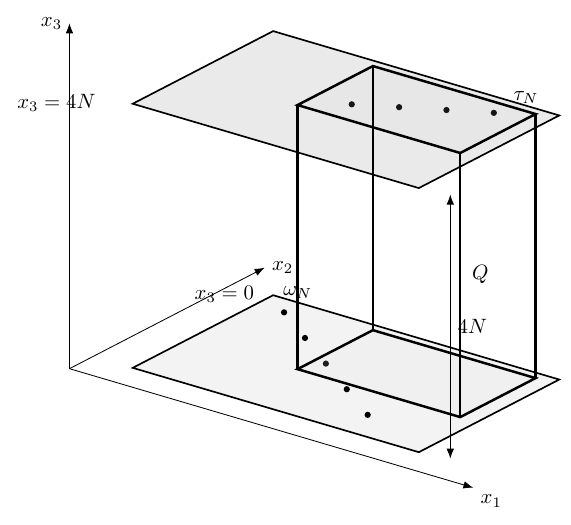}
\caption{The supports of $\omega_N$ and $\tau_N$ on the planes $k=0$ and $k=4N$.}
\label{fig-4}

\end{figure}

\begin{lemma}\label{lem:cube-slice}
If
$$
\omega_N[Q^-_{\Z}]>0,
\qquad
\tau_N[Q^+_{\Z}]>0,
$$
then
\begin{equation*}
4N\le2(r-1),
\qquad\text{and } \quad r>N.
\end{equation*}
Moreover, $1\le a_1,a_2\le N$ and
\begin{equation*}
\begin{aligned}
Q^-_{\Z}\cap(\Lambda_N\times\{0\})
&=
\{1,\ldots,a_1\}\times\{1,\ldots,a_2\}\times\{0\},
\\
Q^+_{\Z}\cap(\Lambda_N\times\{4N\})
&=
\{a_1,\ldots,N\}\times\{a_2,\ldots,N\}\times\{4N\}.
\end{aligned}
\end{equation*}
\end{lemma}

\begin{proof}
By $
\omega_N[Q^-_{\Z}]>0,
\tau_N[Q^+_{\Z}]>0
$, we have
$$
a_3-r+1\le0\le a_3,
\qquad
a_3\le4N\le a_3+r-1.
$$
The first inequality gives $a_3\le r-1$ and the second gives $4N-a_3\le r-1$. Thus $4N\le2(r-1)$ and $r>N$.

The nonempty horizontal intersections imply $1\le a_1,a_2\le N$. Since $r>N$,
$$
a_1-r+1\le a_1-N\le0,
$$
one has
$$\{a_1-r+1,\ldots,a_1\}\cap \{1,\ldots,N\}=\{1,\ldots,a_1\}.$$
The same argument shows that
$$\{a_2-r+1,\ldots,a_2\}\cap \{1,\ldots,N\}=\{1,\ldots,a_2\}.$$
Hence,
\begin{equation*}
\begin{aligned}
Q^-_{\Z}\cap(\Lambda_N\times\{0\})
&=
\{1,\ldots,a_1\}\times\{1,\ldots,a_2\}\times\{0\}.
\end{aligned}
\end{equation*}
Similarly,
\begin{equation*}
\begin{aligned}
Q^+_{\Z}\cap(\Lambda_N\times\{4N\})
&=
\{a_1,\ldots,N\}\times\{a_2,\ldots,N\}\times\{4N\}.
\end{aligned}
\end{equation*}
\end{proof}

If $r>N$, the intersections with the two support planes are shown in Figure~\ref{fig-5}.

\begin{figure}[H]
\centering
\includegraphics[width=.50\textwidth]{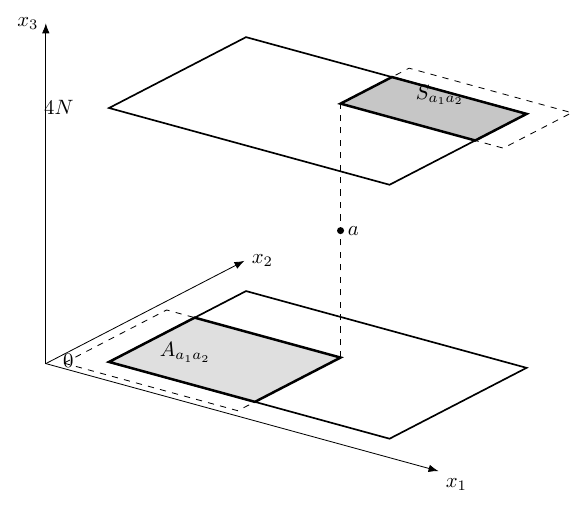}
\caption{The intersections of $Q^-_{\Z}$ and $Q^+_{\Z}$ with the planes supporting $\omega_N$ and $\tau_N$.}
\label{fig-5}

\end{figure}

By Proposition \ref{prop:finite} and Lemma \ref{lem:cube-slice},
\begin{align*}
\omega_N[Q^-_{\Z}]\tau_N[Q^+_{\Z}]^{p-1}
&=
A_{a_1a_2}\bigl((4N)^{3p'}S_{a_1a_2}\bigr)^{p-1}\\
&=(4N)^{3p'(p-1)}=(4N)^{3p}\le 2^{3p}r^{3p}.
\end{align*}
Therefore the estimate is uniform in $N$.

Define
\begin{equation*}
G_N(i,j,k)
=
(4N)^{3p'} g_N(i,j)\nu_{ij}\one_{\{k=4N\}}.
\end{equation*}
\begin{lemma}\label{lem:hardy-embedding}
For every $(i,j)\in\Lambda_N$,
\begin{equation*}
M^+_{\Z^3}G_N(i,j,0)
\ge
\frac{(4N)^{3p'}}{(4N+1)^3}H_Ng_N(i,j),
\end{equation*}
and
\begin{equation*}
\sum_{\{\tau_N>0\}}G_N^p\tau_N^{1-p}
=
(4N)^{3p'}\sum_{i,j}g_N(i,j)^p\nu_{ij}.
\end{equation*}
\end{lemma}

\begin{proof}
Fix $(i,j)\in\Lambda_N$. Consider the following discrete cube
$$
Q_{i,j}
:=
(i,j,0)+\{0,\ldots,4N\}^3.
$$
It contains exactly $(4N+1)^3$ lattice points.

Let $(k,l)\in\Lambda_N$ satisfy
$$
k\ge i,
\qquad
l\ge j.
$$
Since $1\le i,k,l,j\le N$, we have
$$
0\le k-i\le N-i\le N-1\le4N
$$
and
$$
0\le l-j\le N-j\le N-1\le4N.
$$
Hence
$$
(k,l,4N)\in Q_{i,j}
\qquad
\text{whenever }k\ge i,\ l\ge j.
$$
Therefore, by the definition of $M_{\mathbb Z^3}^{+}$ and the
nonnegativity of $g_N$ and $\nu$,
\begin{align*}
M_{\mathbb Z^3}^{+}G_N(i,j,0)
&\ge
\frac1{(4N+1)^3}
\sum_{(x_1,x_2,x_3)\in Q_{i,j}}
G_N(x_1,x_2,x_3)\\
&\ge
\frac1{(4N+1)^3}
\sum_{k=i}^N\sum_{l=j}^N
G_N(k,l,4N).
\end{align*}
By the definition of $G_N$,
\begin{align*}
\sum_{k=i}^N\sum_{l=j}^N
G_N(k,l,4N)
&=
(4N)^{3p'}
\sum_{k=i}^N\sum_{l=j}^N
g_N(k,l)\nu_{kl}\\
&=
(4N)^{3p'}H_Ng_N(i,j).
\end{align*}
Consequently,
$$
M_{\mathbb Z^3}^{+}G_N(i,j,0)
\ge
\frac{(4N)^{3p'}}{(4N+1)^3}
H_Ng_N(i,j).
$$

We next compute the weighted $\ell^p$ norm. Recall that
$$
\tau_N(i,j,k)
=
(4N)^{3p'}\nu_{ij}\,
\mathbf 1_{\{k=4N\}}.
$$
Thus both $G_N$ and $\tau_N$ are supported on the layer
$$
\{(i,j,4N):(i,j)\in\Lambda_N,\ \nu_{ij}>0\}.
$$
At every point of this support,
\begin{align*}
G_N(i,j,4N)^p
\tau_N(i,j,4N)^{1-p}
&=
\bigl((4N)^{3p'}g_N(i,j)\nu_{ij}\bigr)^p
\bigl((4N)^{3p'}\nu_{ij}\bigr)^{1-p}\\
&=
(4N)^{3p'p}
(4N)^{3p'(1-p)}
g_N(i,j)^p
\nu_{ij}^{\,p+1-p}\\
&=
(4N)^{3p'}
g_N(i,j)^p\nu_{ij},
\end{align*}
since
$$
3p'p+3p'(1-p)=3p'.$$
Summing over the support of $\tau_N$ yields
$$
\sum_{\{\tau_N>0\}}
G_N^p\tau_N^{1-p}
=
(4N)^{3p'}
\sum_{(i,j)\in\Lambda_N}
g_N(i,j)^p\nu_{ij},
$$
This completes the proof.
\end{proof}

Lemma~\ref{lem:hardy-embedding}, \eqref{eq:CN-lower}, and \eqref{eq:gN} imply
\begin{equation}\label{eq:finite-lattice-bad}
\frac{\|M^+_{\Z^3}G_N\|_{\ell^{p,\infty}(\omega_N)}}
{\left(\sum_{\{\tau_N>0\}}G_N^p\tau_N^{1-p}\right)^{1/p}}
\ge c_p(\log N)^{1/p'}.
\end{equation}
Hence the $A_p^+$ estimate is uniform in $N$, while the constant in \eqref{eq:finite-lattice-bad} tends to infinity.

To obtain a single pair of weights for which the operator norm is infinite, we combine with increasing values of $N$. Let
$$
N_k=2^{k+3},
\qquad
T_1=0,
\qquad
T_{k+1}=T_k+10N_{k+1},
$$
and
$$
c_k=(T_k,-T_k,0).
$$
Translate the $N_k$-th pair by $c_k$ and denote the resulting measures by $\omega_k$ and $\tau_k$.
Figure~\ref{fig-6} shows the horizontal projections.

\begin{figure}[H]
\centering
\includegraphics[width=.50\textwidth]{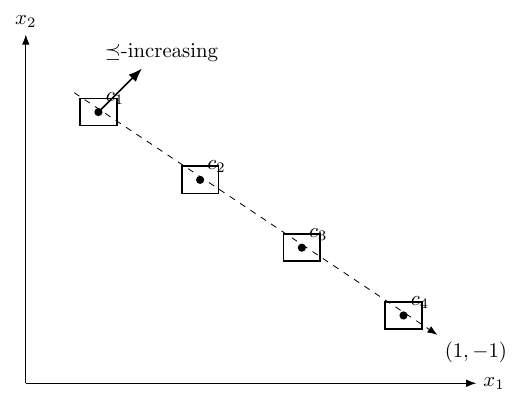}
\caption{Horizontal projections after the translations $c_k=(T_k,-T_k,0)$.}
\label{fig-6}
\end{figure}

The translations satisfy the next separation property.

\begin{lemma}\label{lem:separation-lattice}
With the above choice of $T_k$ and $c_k$,
\begin{equation}\label{eq:block-separation}
x\in\supp\omega_i,
\quad
y\in\supp\tau_j,
\quad x\preceq y
\quad\Longrightarrow\quad i=j.
\end{equation}
\end{lemma}

\begin{proof}
The horizontal projection of the $k$-th block is contained in
$$
[T_k+1,T_k+N_k]\times[-T_k+1,-T_k+N_k].
$$
Suppose first that $i<j$. Since $T_j-T_i\ge10N_j$,
$$
y_2\le -T_j+N_j\le -T_i-9N_j<-T_i+1\le x_2,
$$
so $x_2\le y_2$ is impossible. If $i>j$, then $T_i-T_j\ge10N_i$ and
$$
y_1\le T_j+N_j<T_i+1\le x_1,
$$
so $x_1\le y_1$ is impossible. Then $x\in\supp\omega_i,
y\in\supp\tau_j$ and $x\preceq y$ imply that $i=j$.
\end{proof}

Define
\begin{equation*}
w=\sum_{k=1}^\infty\omega_k,
\qquad
\sigma=1+\sum_{k=1}^\infty\tau_k,
\qquad
v=\sigma^{1-p}.
\end{equation*}
The sums are locally finite. The horizontal projections are disjoint, so $0\le w\le1$. The term $1$ gives $\sigma\ge1$ and hence $v=\sigma^{1-p}>0$.

Fix a pair $Q^-_{\Z},Q^+_{\Z}$. If $Q^+_{\Z}\cap\supp\tau_k=\varnothing$ for every $k$, then
$$
\sigma[Q^+_{\Z}]=r^3,
\qquad
w[Q^-_{\Z}]\le r^3,
$$
so
$$
w[Q^-_{\Z}]\sigma[Q^+_{\Z}]^{p-1}\le r^{3p}.
$$
If $Q^-_{\Z}\cap\supp\omega_i\ne\varnothing$ and $Q^+_{\Z}\cap\supp\tau_j\ne\varnothing$, then \eqref{eq:block-separation} gives $i=j$. Therefore,
$$
w[Q^-_{\Z}]=\omega_k[Q^-_{\Z}],
\qquad
\sigma[Q^+_{\Z}]=r^3+\tau_k[Q^+_{\Z}].
$$
Using
$$
(a+b)^{p-1}\le C_p(a^{p-1}+b^{p-1}),
$$
together with $\omega_k[Q^-_{\Z}]\le r^3$ and
$$
\omega_k[Q^-_{\Z}]\tau_k[Q^+_{\Z}]^{p-1}\le C_pr^{3p},
$$
we obtain
$$
w[Q^-_{\Z}]\sigma[Q^+_{\Z}]^{p-1}\le C_pr^{3p}.
$$
Therefore
$$
A_{p,\Z^3}^+(w,v)<\infty.
$$
For the opposite estimate, set
$$
\mathcal G_k(x)=G_{N_k}(x-c_k).
$$
On $\supp\mathcal G_k$ one has $\sigma\ge\tau_k$. Since $1-p<0$,
$$
v=\sigma^{1-p}\le\tau_k^{1-p}.
$$
Moreover, $w\ge\omega_k$. Using \eqref{eq:finite-lattice-bad},
$$
\frac{\|M^+_{\Z^3}\mathcal G_k\|_{\ell^{p,\infty}(w)}}
{\|\mathcal G_k\|_{\ell^p(v)}}
\ge c_p(\log N_k)^{1/p'}\longrightarrow\infty.
$$
This proves the second assertion of Theorem~\ref{thm:lattice}.

\section{Continuous finite examples}

We now pass from the lattice result to $\R^3$. All geometric scales in this section are fixed once $N$ is chosen. Set
$$
D=[0,1/8]^2,\qquad I^-=[-1/8,0],\qquad I^+=[1,9/8],
$$
and let
$$
\rho_N=\frac{1}{8N}.
$$
Thus $D$ is divided into $N^2$ cubes of side length $\rho_N$. Inside the $(i,j)$-cube we place two smaller squares of side length $\rho_N/10$:
\begin{align*}
L_{ij}
&=\bigl((i-1)\rho_N,(i-1)\rho_N+\rho_N/10\bigr]
  \times\bigl((j-1)\rho_N,(j-1)\rho_N+\rho_N/10\bigr],\\
U_{ij}
&=\bigl(i\rho_N-\rho_N/10,i\rho_N\bigr]
  \times\bigl(j\rho_N-\rho_N/10,j\rho_N\bigr].
\end{align*}
This separation converts the inequalities $k\ge i$ and $\ell\ge j$ into geometric inclusion relations, see Figure \ref{fig-7}.

\begin{figure}[H]
\centering
\includegraphics[width=.38\textwidth]{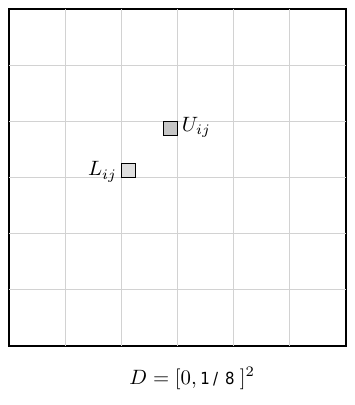}
\caption{The sets $L_{ij}$ and $U_{ij}$ inside a cube of side length $\rho_N$.}
\label{fig-7}
\end{figure}

Define
$$
W_N(x')=\Bigl(\frac{10}{\rho_N}\Bigr)^2
\sum_{i,j}\mu_{ij}\one_{L_{ij}}(x'),
\qquad
\Sigma_N(x')=\Bigl(\frac{10}{\rho_N}\Bigr)^2
\sum_{i,j}\nu_{ij}\one_{U_{ij}}(x').
$$
Since $|L_{ij}|=|U_{ij}|=(\rho_N/10)^2$,
$$
\int_{L_{ij}}W_N=\mu_{ij},
\qquad
\int_{U_{ij}}\Sigma_N=\nu_{ij}.
$$
Set
$$
\widetilde W_N(x',x_3)=8W_N(x')\one_{I^-}(x_3),
\qquad
\widetilde\Sigma_N(x',x_3)=8\Sigma_N(x')\one_{I^+}(x_3).
$$
Then,
$$
\int_{L_{ij}\times I^-}\widetilde W_N=\mu_{ij},
\qquad
\int_{U_{ij}\times I^+}\widetilde\Sigma_N=\nu_{ij}.
$$
Figure \ref{fig-8} shows the supports of $\widetilde W_N$ and $\widetilde\Sigma_N$.

\begin{figure}[H]
\centering
\includegraphics[width=.50\textwidth]{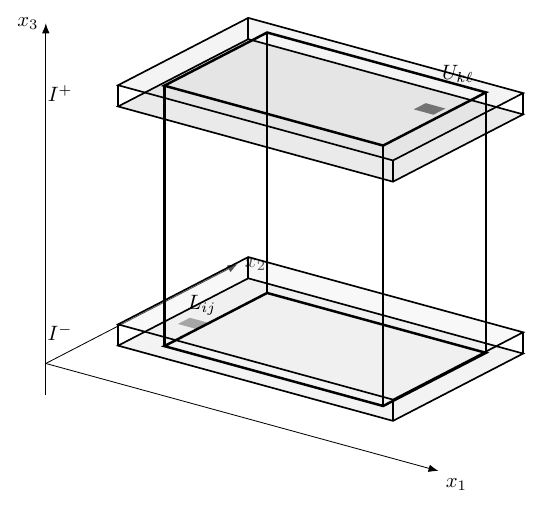}
\caption{The supports of $\widetilde W_N$ and $\widetilde\Sigma_N$.}
\label{fig-8}
\end{figure}

\begin{lemma}\label{lem:continuous-localization}
Let
$$
Q=\prod_{r=1}^3[a_r-h,a_r],
\qquad
Q^+=\prod_{r=1}^3[a_r,a_r+h].
$$
If $\widetilde W_N(Q)>0$ and $\widetilde\Sigma_N(Q^+)>0$, then $h\ge1/2$. Moreover, if
$$
I=\left\lceil\frac{a_1}{\rho_N}\right\rceil,
\qquad
J=\left\lceil\frac{a_2}{\rho_N}\right\rceil,
$$
then
$$
\widetilde W_N(Q)\le A_{IJ},
\qquad
\widetilde\Sigma_N(Q^+)\le S_{IJ}.
$$
\end{lemma}

\begin{proof}
By $\widetilde W_N(Q)>0$ and $\widetilde\Sigma_N(Q^+)>0$, there exist $s_-\in[-1/8,0]$ and $s_+\in[1,9/8]$ such that
$$
a_3-h\le s_-\le a_3\le s_+\le a_3+h.
$$
Thus $1\le s_+-s_-\le2h$, and therefore $h\ge1/2$.

Because $\widetilde W_N(Q)$ and $\widetilde\Sigma_N(Q^+)$ are positive, the common horizontal vertex $(a_1,a_2)$ lies in $D$. If $L_{ij}$ meets the horizontal projection of $Q$, then $(i-1)\rho_N<a_1$ and $(j-1)\rho_N<a_2$. Hence $i\le I$ and $j\le J$. Summing the integrals of $\widetilde W_N$ over the relevant boxes gives
$$
\widetilde W_N(Q)\le\sum_{i\le I}\sum_{j\le J}\mu_{ij}=A_{IJ}.
$$
Similarly, if $U_{k\ell}$ meets the horizontal projection of $Q^+$, then $k\rho_N\ge a_1$ and $\ell\rho_N\ge a_2$. Thus $k\ge I$ and $\ell\ge J$, and
$$
\widetilde\Sigma_N(Q^+)\le\sum_{k\ge I}\sum_{\ell\ge J}\nu_{k\ell}=S_{IJ}.
$$
\end{proof}

\begin{proposition}\label{prop:continuous-finite}
For every pair $Q,Q^+$,
\begin{equation}\label{eq:continuous-block-Ap}
\widetilde W_N(Q)\widetilde\Sigma_N(Q^+)^{p-1}
\le2^{3p}|Q|^p.
\end{equation}
Moreover, there exists a nonnegative compactly supported function $F_N$, supported in $\supp\widetilde\Sigma_N$, such that
\begin{equation}\label{eq:continuous-block-bad}
\frac{\|M_3^+F_N\|_{L^{p,\infty}(\widetilde W_N)}}
{\left(\int_{\R^3}F_N^p\widetilde\Sigma_N^{1-p}\right)^{1/p}}
\ge c_p(\log N)^{1/p'}.
\end{equation}
\end{proposition}

\begin{proof}
If one of the $\widetilde W_N(Q), \widetilde\Sigma_N(Q^+)$ in \eqref{eq:continuous-block-Ap} is zero, there is nothing to prove. Otherwise Lemma~\ref{lem:continuous-localization} gives $h\ge1/2$ and
$$
\widetilde W_N(Q)\widetilde\Sigma_N(Q^+)^{p-1}
\le A_{IJ}S_{IJ}^{p-1}=1.
$$
Since $|Q|=h^3$ and $h\ge1/2$,
$$
\widetilde W_N(Q)\widetilde\Sigma_N(Q^+)^{p-1}\leq 1\le2^{3p}|Q|^p.
$$

Choose $g_N$ as above and define
$$
F_N(y',y_3)
=
8\Bigl(\frac{10}{\rho_N}\Bigr)^2
\sum_{k,\ell}g_N(k,\ell)\nu_{k\ell}
\one_{U_{k\ell}}(y')\one_{I^+}(y_3).
$$
Fix $x'\in L_{ij}$. Since the side length of $L_{ij}$ and $U_{ij}$ is $\rho_N/10$,
$$
x_1\le(i-1)\rho_N+\frac{\rho_N}{10}
<i\rho_N-\frac{\rho_N}{10},
$$
and the same inequality holds in the second coordinate. Consequently,
$$
k\ge i,\ \ell\ge j
\quad\Longrightarrow\quad
U_{k\ell}\subset[x_1,1/8]\times[x_2,1/8].
$$
For $x_3\in I^-$, the set
$$
[x_1,1/8]\times[x_2,1/8]\times I^+
$$
is contained in the cube $(x',x_3)+[0,5/4]^3$. Hence
$$
M_3^+F_N(x',x_3)
\ge
\Bigl(\frac45\Bigr)^3
H_Ng_N(i,j).
$$
Therefore
$$
\|M_3^+F_N\|_{L^{p,\infty}(\widetilde W_N)}
\ge
\Bigl(\frac45\Bigr)^3
\|H_Ng_N\|_{\ell^{p,\infty}(\mu)}.
$$

On $U_{k\ell}\times I^+$,
$$
F_N=8\Bigl(\frac{10}{\rho_N}\Bigr)^2g_N(k,\ell)\nu_{k\ell},
\qquad
\widetilde\Sigma_N=8\Bigl(\frac{10}{\rho_N}\Bigr)^2\nu_{k\ell}.
$$
Multiplying by the volume $(\rho_N/10)^2/8$ gives
$$
\int_{U_{k\ell}\times I^+}F_N^p\widetilde\Sigma_N^{1-p}
=g_N(k,\ell)^p\nu_{k\ell}.
$$
It follows that
$$
\int_{\R^3}F_N^p\widetilde\Sigma_N^{1-p}
=
\sum_{k,\ell}g_N(k,\ell)^p\nu_{k\ell}.
$$
The lower bound now follows from the choice of $g_N$ and the estimate
$$
C_N\ge c_p(\log N)^{1/p'}.
$$
\end{proof}

The final step requires a normalization of each finite continuous pair. Let $\alpha>0$ and set
\begin{equation}\label{eq:scaling}
\beta=\alpha^{-1/(p-1)},\qquad
W_N^{(\alpha)}=\alpha\widetilde W_N,\qquad
\Sigma_N^{(\alpha)}=\beta\widetilde\Sigma_N,\qquad
F_N^{(\alpha)}=\beta F_N.
\end{equation}
Since $\alpha\beta^{p-1}=1$, Proposition~\ref{prop:continuous-finite} shows that
$$
\frac{
\|M_3^+F_N^{(\alpha)}\|_{L^{p,\infty}(W_N^{(\alpha)})}
}{
\left(\int (F_N^{(\alpha)})^p(\Sigma_N^{(\alpha)})^{1-p}\right)^{1/p}
}
=
\frac{
\|M_3^+F_N\|_{L^{p,\infty}(\widetilde W_N)}
}{
\left(\int F_N^p\widetilde\Sigma_N^{1-p}\right)^{1/p}
}.
$$
We shall choose $\alpha$ so that $W_N^{(\alpha)}\le1$.

\section{Proof of the main theorem}
\subsection{Dimension $d=3$}
The continuous finite estimates are uniform in $N$. Choose $N_k=2^{k+3}$. For each $k$, take the finite continuous pair associated with $N_k$. Define
\begin{equation*}
\alpha_k
=
\min\left\{1,\|\widetilde W_{N_k}\|_\infty^{-1}\right\},
\qquad
\beta_k=\alpha_k^{-1/(p-1)}.
\end{equation*}
Then
$$
0\le\alpha_k\widetilde W_{N_k}\le1.
$$
Fix $R>1/8$ and set
$$
c_k=(Rk,-Rk,0).
$$
Define
\begin{equation*}
\omega_3^k(x)=\alpha_k\widetilde W_{N_k}(x-c_k),
\qquad
\tau_3^k(x)=\beta_k\widetilde\Sigma_{N_k}(x-c_k).
\end{equation*}
Since every horizontal support is contained in a translate of $D=[0,1/8]^2$ and $R>1/8$,
\begin{equation}\label{eq:continuous-separation}
x\in\supp\omega_i,
\quad y\in\supp\tau_j,
\quad x\preceq y
\quad\Longrightarrow\quad i=j.
\end{equation}
If $i<j$, then
$$
y_2\le -Rj+1/8<-Ri\le x_2,
$$
so $x_2\le y_2$ is impossible. If $i>j$, the first coordinate gives $y_1<x_1$. Hence \eqref{eq:continuous-separation} holds.

Define
\begin{equation*}
w_3=\sum_{k=1}^\infty\omega_3^k,
\qquad
\sigma_3=1+\sum_{k=1}^\infty\tau_3^k,
\qquad
v_3=\sigma_3^{1-p}.
\end{equation*}
The sums are locally finite because $|c_k|\to\infty$ and all horizontal supports have diameter at most $\sqrt2/8$. Moreover,
$$
0\le w_3\le1,
\qquad
\sigma_3\ge1,
\qquad
0<v_3=\sigma_3^{1-p}\le1.
$$
Thus $w_3$ and $v_3$ are locally integrable.

If $Q_3$ meets $\supp\omega_3^i$ and $Q_3^+$ meets $\supp\tau_3^j$, then every $x\in Q$ and $y\in Q^+$ satisfy $x\preceq y$, and \eqref{eq:continuous-separation} gives $i=j$. If $Q_3^+\cap\supp\tau_k=\varnothing$ for every $k$, then
$$
\sigma(Q_3^+)=|Q_3|,
\qquad
w_3(Q_3)\le|Q_3|,
$$
so
$$
w_3(Q_3)\sigma(Q_3^+)^{p-1}\le|Q_3|^p.
$$
If both sides meet the same block $k$, then
$$
w_3(Q_3)=\omega_3^k(Q_3),
\qquad
\sigma_3(Q_3^+)=|Q_3|+\tau_3^k(Q_3^+).
$$
Using $(a+b)^{p-1}\le C_p(a^{p-1}+b^{p-1})$, $w\le1$, Proposition~\ref{prop:continuous-finite}, and the scaling invariance \eqref{eq:scaling}, we obtain
\begin{align*}
w_3(Q_3)\sigma_3(Q_3^+)^{p-1}
&\le
C_p\omega_3^k(Q_3)|Q_3|^{p-1}
+C_p\omega_3^k(Q_3)\tau_3^k(Q_3^+)^{p-1}\\
&\le C_p|Q_3|^p.
\end{align*}
Therefore
\begin{equation*}
A_{p,3}^+(w,v)<\infty.
\end{equation*}

Let $F_{N_k}$ be the function from Proposition~\ref{prop:continuous-finite} and set
\begin{equation*}
f_k(x)=\beta_kF_{N_k}(x-c_k).
\end{equation*}
On $\supp f_k$ we have $\sigma_3\ge\tau_3^k$. Since $1-p<0$,
$$
v_3=\sigma_3^{1-p}\le(\tau_3^k)^{1-p}.
$$
Hence
\begin{equation*}
\int_{\R^3}f_k^pv_3
\le
\int_{\R^3}f_k^p(\tau_3^k)^{1-p}.
\end{equation*}
Also $w_3\ge\omega_3^k$, so
\begin{equation*}
\|M_3^+f_k\|_{L^{p,\infty}(w_3)}
\ge
\|M_3^+f_k\|_{L^{p,\infty}(\omega_3^k)}.
\end{equation*}
It follows that
$$
\frac{\|M_3^+f_k\|_{L^{p,\infty}(w_3)}}{\|f_k\|_{L^p(v_3)}}
\ge
c_p(\log N_k)^{1/p'}\longrightarrow\infty.
$$
Thus
$$
\|M_3^+\|_{L^p(v_3)\to L^{p,\infty}(w_3)}=\infty.
$$
This proves the three-dimensional case of Theorem~\ref{thm:main}.

\subsection{Extension to dimension $d>3$}

We now show that the three-dimensional counterexample extends to every higher dimension.

\begin{proposition}\label{prop:higher-dimensional-extension}
Suppose that weights $w_3,v_3$ on $\mathbb R^3$ satisfy
$$
A_{p,3}^+(w_3,v_3)<\infty
$$
and
$$
\|M_3^+\|_{L^p(v_3)\to L^{p,\infty}(w_3)}=\infty.
$$
Then, for every $d>3$, there are weights $w_d,v_d$ on $\mathbb R^d$
such that
$$
A_{p,d}^+(w_d,v_d)=A_{p,3}^+(w_3,v_3)
$$
and
$$
\|M_d^+\|_{L^p(v_d)\to L^{p,\infty}(w_d)}=\infty.
$$
\end{proposition}

\begin{proof}
Write
$$
x=(x',x'')\in\mathbb R^3\times\mathbb R^{d-3}.
$$
Set
$$
w_d(x',x'')=w_3(x'),
\qquad
v_d(x',x'')=v_3(x').
$$
If $\sigma_j=v_j^{-1/(p-1)}$, then
$\sigma_d(x',x'')=\sigma_3(x')$.

Consider a one-sided pair
$$
Q_d^-=\prod_{j=1}^d[a_j-h,a_j],
\qquad
Q_d^+=\prod_{j=1}^d[a_j,a_j+h],
$$
and let $Q_3^\pm$ denote its projection onto the first three
coordinates. Fubini's theorem gives
$$
w_d(Q_d^-)=h^{d-3} w_3(Q_3^-),
\qquad
\sigma_d(Q_d^+)=h^{d-3}\sigma_3(Q_3^+).
$$
Since $|Q_d^-|=h^d$,
$$
\frac{w_d(Q_d^-)\sigma_d(Q_d^+)^{p-1}}{|Q_d^-|^p}
=
\frac{w_3(Q_3^-)\sigma_3(Q_3^+)^{p-1}}{h^{3p}}.
$$
Taking the supremum over $a$ and $h$ yields
$$
A_{p,d}^+(w_d,v_d)=A_{p,3}^+(w_3,v_3).
$$

Assume that
$$
\|M_d^+F\|_{L^{p,\infty}(w_d)}
\le C\|F\|_{L^p(v_d)}
$$
for every compactly supported $F$. Let
$f\in L^p(v_3)$ be compactly supported and, for $R>0$, define
$$
F_R(x',x'')
=
R^{-(d-3)/p}f(x')\mathbf 1_{[0,R]^{d-3}}(x'').
$$
Then
$$
\|F_R\|_{L^p(v_d)}=\|f\|_{L^p(v_3)}.
$$
Let
$$
M_{3,R/2}^+f(x')
=
\sup_{0<h\le R/2}
\frac1{h^3}\int_{x'+[0,h]^3}|f(y')|\,dy'.
$$
For $x''\in[0,R/2]^{d-3}$ and $0<h\le R/2$,
$$
x''+[0,h]^{d-3}\subset[0,R]^{d-3},
$$
and
$$
M_d^+F_R(x',x'')
\ge
R^{-(d-3)/p}M_{3,R/2}^+f(x').
$$
Consequently, for every $\lambda>0$,
\begin{align*}
w_d\bigl(\{M_d^+F_R>\lambda\}\bigr)
&\ge
\left(\frac R2\right)^{d-3}
w_3\bigl(
\{M_{3,R/2}^+f>R^{(d-3)/p}\lambda\}
\bigr).
\end{align*}
After the change of variable
$t=R^{(d-3)/p}\lambda$, we obtain
$$
\|M_d^+F_R\|_{L^{p,\infty}(w_d)}
\ge
2^{-(d-3)/p}
\|M_{3,R/2}^+f\|_{L^{p,\infty}(w_3)}.
$$
Therefore,
$$
\|M_{3,R/2}^+f\|_{L^{p,\infty}(w_3)}
\le
2^{(d-3)/p}C\|f\|_{L^p(v_3)}.
$$
As $R\to\infty$,
$$
M_{3,R/2}^+f(x')\uparrow M_3^+f(x').
$$
Hence, for every $\lambda>0$,
the monotone convergence gives
$$
\|M_3^+f\|_{L^{p,\infty}(w_3)}
\le
2^{(d-3)/p}C\|f\|_{L^p(v_3)}.
$$
By a standard truncation argument,
the same estimate extends to all of
$L^p(v_3)$. This contradicts the choice of $w_3,v_3$.
\end{proof}

Proposition~\ref{prop:higher-dimensional-extension}, together with the
three-dimensional construction, completes the proof of
Theorem~\ref{thm:main}.

\section{Acknowledgement}

Dinghuai Wang is supported by National Natural Science Foundation of China (No.~12101010), Natural Science Foundation of Anhui Province (No.
2108085QA19) and Key Scientific Project of Higher Education Institutions in Anhui Province (No.2023AH050145).

\vspace{0.5cm}

\vskip 0.2 true cm
{\bf Conflict of Interest Statement:}

\vskip 0.2 true cm

{The authors declare that there is no conflict of interest in relation to this article.}

\vskip 0.2 true cm
{\bf Data availability statement:}

\vskip 0.2 true cm

{Data sharing is not applicable to this article as no data sets are generated
during the current study.}

\vskip 0.2 true cm

\end{document}